\documentclass[12pt]{article}
\usepackage{amsmath, amsthm}
\usepackage[top=4cm,left=3cm,right=3cm,bottom=4cm,headheight=13.6pt]{geometry}\usepackage{amsfonts}
\usepackage{amssymb}
\usepackage{fancyhdr}
\usepackage{enumerate}
\usepackage{amsmath,amsthm}
\usepackage[dvips]{graphicx}
\usepackage{color}
\usepackage{eepic}
\usepackage{epic}
\usepackage[dvips]{rotating}
\usepackage{sectsty}
\usepackage{wrapfig}

\theoremstyle{plain}
\newtheorem{theorem}{Theorem}[section]
\newtheorem{lemma}[theorem]{Lemma}

\theoremstyle{definition}

\newtheorem{remark}[theorem]{Remark}

\theoremstyle{remark}

 \numberwithin{equation}{section}

\begin{document}
\title{On Blow-up  of  a Seimilinear Heat Equation with Nonlinear Boundary Conditions}
\author{Maan A. Rasheed and Miroslav Chlebik}
\maketitle

\abstract 
This paper deals with the blow-up properties of the solutions of the semilinear heat equation $u_t=\Delta u+\lambda e^{pu}$ in $B_R \times (0,T)$ with the nonlinear boundary conditions 
$\frac{\partial u}{\partial \eta}=e^{qu}$ on $\partial B_R \times (0,T),$ where $B_R$ is a ball in $R^n,$ $\eta$ is the outward normal,   $p>0, q>0,~\lambda >0.$ The upper and lower blow-up rate estimates are established. It is also proved under some restricted assumptions, that the blow-up occurs only on the boundary. 

 \section{Introduction}
 
In this paper, we consider the initial-boundary value problem

\begin{equation}\label{B9} \left.
\begin{array}{ll}
u_t=\Delta u+\lambda e^{pu}, &\quad (x,t)\in B_R \times (0,T), \\ 
\frac{\partial u}{\partial \eta}=e^{qu},&\quad (x,t) \in\partial B_R \times (0,T),\\
u(x,0)=u_0(x),& \quad x\in B_R,\end{array} \right \} \end{equation} where $p>0, q>0,~\lambda >0,$  $B_R$ is a ball in $R^n,$  $\eta$ is the outward normal, $u_0$ is nonnegative, radially symmetric, nondecreasing, smooth function satisfies the conditions
\begin{eqnarray}\label{B4}
&&\frac{\partial u_0}{\partial \eta}=e^{qu_0}, \qquad x\in \partial \Omega,\\
&&\label{B4a}  \Delta u_0+\lambda e^{pu_0}\ge 0,\quad u_{0r}(|x|)\ge 0, \quad x\in \overline{\Omega}_R.\end{eqnarray}

The problem of the semilinear heat equation with nonlinear boundary conditions:

\begin{equation}\label{B1} \left.
\begin{array}{ll}
u_t=\Delta u+\lambda f(u), & \quad (x,t)\in \Omega \times (0,T), \\  \frac{\partial u}{\partial \eta}=g(u),&  \quad (x,t)\in \partial \Omega \times (0,T),\\
 u(x,0)=u_0(x),&\quad x\in {\Omega}, \end{array} \right\} \end{equation}
 
 has been studied by many authors (see for example \cite{49,50,51}). The crucial point of these works was the question whether the reaction term in the semilinear equation can prevent (affect) blow-up. For instance, in \cite{49} it has been studied the blow-up solutions of problem (\ref{B1}), where $\lambda<0$ and \begin{equation}\label{B5} f(u)=u^p,\quad g(u)=u^q, \quad p,q>1,\end{equation} for $n=1$ or $\Omega=B_R.$  Particularly, it was shown that the exponent $p=2q-1$ is critical for blow-up in the following sense:
\begin{enumerate}[(i)]
\item If $p<2q-1$ (or $p=2q-1$ and $-\lambda <q$), then there exist solutions, which blow up in finite time and the blow-up occurs only on the boundary.
\item If $p>2q-1$ (or $p=2q-1$ and $-\lambda >q$), then all solutions exist globally and are globally bounded.
\end {enumerate}   In \cite{31} J. D. Rossi has proved for the case (i), where $n=1,$ $\Omega =[0,1],$ that there exist positive constants $C,c$ such that the upper (lower) blow-up rate estimate take the following forms
  $$ c\le\max_{[0,1]}u(\cdot,t)(T-t)^{\frac{1}{2(q-1)}} \le C,\quad  0<t<T.$$

In \cite{51} it has been studied another special case of problem (\ref{B1}), where $\lambda=1,$ $f,g$ as in (\ref{B5}), $\Omega=[0,1]$ or it is a bounded domain with $C^2$ boundary, it was proved that the solutions of (\ref{B1}) exist globally if and only if $\max\{p,q\}\le 1,$ otherwise, every solution has to blow up in finite time. Moreover, the blow-up occurs only on the boundary. The blow-up rate estimate for this case has been studied in \cite{51,31}, for $n=1,\Omega=[0,1],$ it has been shown that there exist positive constants $c,C$ such that 
$$ c\le \max_{[0,1]}u(\cdot,t)(T-t)^{\alpha} \le C, \quad 0<t<T,$$ where $\alpha=1/(p-1)$ if $p \ge 2q-1,$ and $\alpha=1/[2(q-1)]$ if $p<2q-1.$ 

We observe that if $p<2q-1,$ then the nonlinear term at the boundary determines and gives the blow-up rate while, if $p>2q-1,$ then the reaction term in the semilinear equation dominates and gives the blow-up rate.

 Later, in \cite{50} it was considered a second  special case of (\ref{B1}), where $\lambda=-a, a>0,$ $f,g$ are of exponential forms, namely
\begin{equation}\label{B6} \left.
\begin{array}{ll}
u_t=\Delta u-ae^{pu},&\quad (x,t)\in\Omega \times (0,T), \\
 \frac{\partial u}{\partial \eta}=e^{qu},&\quad (x,t) \in\partial \Omega \times (0,T),\\
u(x,0)=u_0(x),& \quad x\in {\Omega}, \end{array} \right\} \end{equation}
where  $p,q >0,$ $u_0$ satisfies (\ref{B4}), (\ref{B4a}). 

It has been shown that in case of $\Omega$ is a bounded domain with smooth boundary, the critical exponent can be given as follows  
\begin{enumerate}[\rm(i)]
\item If $2q<p,$ the solutions of problem (\ref{B6}) are globally bounded.
\item If $2q>p,$ the solutions  of problem (\ref{B6}) blow up in finite time for large initial data.
\item If $2q=p,$ the solutions may blow up in finite time for large initial data. 
\end{enumerate}
Moreover, in case $\Omega=B_R,$ the blow-up occurs only on the boundary and there exist positive constants $c,C$ such that the upper (lower) blow-up rate estimate take the following form
$$\log C_1- \frac{1}{2q}\log(T-t)\le \max_{\overline{B}}u(\cdot,t)\le \log C_2- \frac{1}{2q}\log(T-t),\quad 0<t<T.$$
Therefore, the blow-up properties (blow-up location and bounds) of problem (\ref{B6}) are the same as that of problem (\ref{B6}), where $a=0,$ which has been considered in \cite{17}. 

In this paper, we study the blow-up solutions of problem (\ref{B9}). The upper (lower) blow-up rate estimates is obtained. Moreover, under some restricted assumptions, we prove that blow-up occurs only on the boundary.

\section{Preliminaries} 
Since $f(u)=\lambda e^{pu},~g(u)=e^{qu}$ are smooth functions, and problem (\ref{B9}) is uniformly parabolic, also $u_0$ satisfies the compatibility condition (\ref{B4}), it follows that the existence and uniqueness of local classical solutions to problem (\ref{B9}) are known by the standard theory \cite{37}. On the other hand, the nontrivial solutions of this problem blow up in finite time and the blow-up set contains $\partial B_R,$ and that due to comparison principle, \cite{21}, and the known blow-up result of problem (\ref{B9}), where $\lambda=0$ (see\cite{17}).

In this paper, we denote for simplicity $u(x,t)=u(r,t).$  The following lemma shows some properties of the classical solutions to problem (\ref{B9}). 
\begin{lemma}\label{sugar} Let $u$ be a classical solution to problem (\ref{B9}), where $u_0$ satisfies the assumptions (\ref{B4}), (\ref{B4a}).Then
\begin{enumerate}[\rm(i)]
\item$u> 0,$  radial  in  $\overline{B}_R\times (0,T).$ 
\item $u_r \ge 0,$ in $[0,R]\times [0,T).$
\item  $u_t >0$ in $\overline{B}_R \times (0,T).$
\end{enumerate}
\end{lemma}

\section{Blow-up Rate Estimates}
Since $u_r \ge 0,$ in $[0,R] \times(0,T),$ 
it follows that $$\max_{\overline{B}_R}u(\cdot,t)=u(R,t),\quad 0<t<T.$$ Therefore, it is sufficient to derive the upper (lower) bounds of blow-up rate for $u(R,t).$
\begin{theorem}\label{pole}
Let $u$ be a solution to problem (\ref{B9}), where $u_0$ satisfies the assumptions (\ref{B4}), (\ref{B4a}), $T$ is the blow-up time.Then there is a positive constant $c$ such that
$$\log c-\frac{1}{2\alpha}\log(T-t)\le u(R,t),\quad t \in (0,T),$$ where $\alpha=\max\{p,q\}.$ \end{theorem} 
\begin{proof}
Define $$M(t)=\max_{\overline{B}_R}u(\cdot,t)=u(R,t),\quad \mbox{for}\quad t \in [0,T).$$ Clearly, $M(t)$ is increasing in $(0,T)$ (due to $u_t> 0,$ for $t \in (0,T),~ x \in \overline B_R$). As in \cite{50}, for $0< z<t<T, x \in B_R,$  the integral equation of problem (\ref{B9}) with respect to $u,$ can be written as follows 
\begin{eqnarray}\label{Sunday} u(x,t)&=&\int_{B_R} \Gamma(x-y,t-z)u(y,z)dy+\lambda\int_z^t \int_{B_R} \Gamma(x-y,t-\tau)e^{pu(y,\tau)}dy d\tau \nonumber \\  &&+\int_z^t \int_{S_R} \Gamma(x-y,t-\tau)e^{qu(y,\tau)}ds_y d\tau \nonumber\\ &&-\int_z^t \int_{S_R} u(y,\tau)\frac{\partial \Gamma}{\partial \eta_y}(x-y,t-\tau)ds_y d\tau,\end{eqnarray}
where $\Gamma$ is the fundamental solution of the heat equation, namely 
\begin{equation}\label{op}\Gamma(x,t)=\frac{1}{(4\pi t)^{(n/2)}}\exp[-\frac {|x|^2}{4t}] .\end{equation} 
Since $u(y,t)\le u(R,t) $ for $y\in \overline{B}_R,$  so, the last equation becomes
\begin{eqnarray*}u(x,t)&\le& u(R,z) \int_{B_R} \Gamma(x-y,t-z)dy+\lambda\int_z^t e^{pu(R,\tau)}\int_{B_R} \Gamma(x-y,t-\tau)dy d\tau .\\
 &&+\int_z^t e^{qu(R,\tau)} \int_{S_R} \Gamma(x-y,t-\tau)ds_y d\tau\\ &&+\int_z^t u(R,\tau) \int_{S_R}\mid \frac{\partial \Gamma}{\partial \eta_y}(x-y,t-\tau)\mid ds_y d\tau.\end{eqnarray*} 
Since u is a continuous function on $\overline{B}_R,$ the last inequality leads to  
\begin{eqnarray}\label{B12} M(t)&\le& M(z) \int_{B_R} \Gamma(x-y,t-z)dy+\lambda e^{pM(t)} \int_z^t \int_{B_R} \Gamma(x-y,t-\tau)dy d\tau\nonumber \\
 &&+e^{qM(t)}  \int_z^t \int_{S_R} \Gamma(x-y,t-\tau)ds_y d\tau \nonumber \\ &&+M(t)\int_z^t  \int_{S_R}\mid \frac{\partial \Gamma}{\partial \eta_y}(x-y,t-\tau)\mid ds_y d\tau.\end{eqnarray}  
It is known from  \cite{23,21} that for $0<t_1<t_2,~x,y \in R^n,$ $\Gamma$ satisfies
$$ \int_{B_R} \Gamma(x-y,t_2-t_1)dy \le 1.$$
Moreover, there exist positive constants $k_1,k_2$ such that   
\begin{eqnarray*}&&\Gamma(x-y,t_2-t_1) \le \frac{k_1}{{(t_2-t_1)}^{\mu_0}}\cdot\frac{1}{|x-y|^{n-2+\mu_0}},\quad  0<\mu_0<1,\\
&& \mid \frac{\partial \Gamma}{\partial \eta_{y}}(x-y,t_2-t_1)\mid \le \frac{k_2}{(t_2-t_1)^{\mu}}\cdot\frac{1}{|x-y|^{n+1-2\mu-\sigma}}, \quad \sigma\in (0,1),~ \mu \in (1-\frac{\sigma}{2},1).\end{eqnarray*}

If we choose $\mu_0=1/2,$ then from \cite{23},  there exist positive constants $d_1,d_2$ such that 
$$\int_{S_R}\frac{ds_y}{|x-y|^{n-2+\mu_0}} \le d_1,\qquad \int_{S_R} \frac{ds_y}{|x-y|^{n+1-2\mu-\sigma}}\le d_2.$$
From above it follows that there exist $C_1,C_2>0$ such that, the inequality (\ref{B12}) becomes
$$M(t)\le M(z)+\lambda e^{pM(t)}(t-z)+C_1e^{qM(t)}\sqrt{t-z}+C_2M(t)(t-z)^{1-\mu}.$$ 
Since $t-z \le T-z,$ it follows that \begin{equation}\label{tmr}M(t)\le M(z)+\lambda e^{pM(t)}\sqrt{T-z}+C_1e^{qM(t)}\sqrt{T-z}+C_2M(t)(T-z)^{1-\mu},\end{equation} 
provided $(T-z)\le 1.$ 

Clearly, 
$$\frac{M(t)}{e^{\alpha M(t)}}\longrightarrow 0,\quad \mbox{when}\quad t\rightarrow T.$$ Thus $$\frac{M(t)}{e^{\alpha M(t)}} \le (T-z)^{\frac{1}{2}-(1-\mu)},~ \mbox{for}~ t~\mbox{close to}~ T.$$ 

Therefore, the inequality (\ref{tmr}) becomes  $$M(t)\le M(z)+\lambda e^{pM(t)}\sqrt{T-z}+C_1e^{qM(t)}\sqrt{T-z}+C_2e^{ \alpha M(t)}\sqrt{T-z},$$ 
 
thus there is a constant $C^*$ such that
$$M(t)\le M(z)+C^*e^{\alpha M(t)}\sqrt{T-z}, \quad z<t<T,~t~\mbox{close to}~ T.$$ 
For any $z$ close to $T,$ we can choose $z<t<T$ such that $$M(t)-M(z)=C_0>0,$$ 
which implies $$C_0\le C^*e^{\alpha M(z)+\alpha C_0}\sqrt{T-z}.$$
Thus
$$\frac{C_0}{C^*e^{(\alpha C_0)}\sqrt{T-z}}\le e^{\alpha u(R,z)}.$$ Therefore, there exist  a positive constant $c$ such that $$\log c-\frac{1}{2\alpha}\log(T-t)\le u(R,t),\quad t \in (0,T).$$
\end{proof}

The next theorem shows similar results to Theorem \ref{pole} with  adding more restricted assumptions on $q$ and $u_0.$ The proof relies on the maximum principle rather than the integral equation. 

\begin{theorem}\label{pole1}
Let $u$ be a solution to problem (\ref{B9}), where $q\ge1,$ $T$ is the blow-up time, $u_0$ satisfies the assumptions (\ref{B4}), (\ref{B4a}), moreover, it satisfies the following condition
\begin{equation}\label{FH} u_{0r}(r)-\frac{r}{R}e^{u_0(r)}\ge 0, \quad r\in [0,R].\end{equation} Then there is a positive constant $c$ such that
$$\log c-\frac{1}{2\alpha}\log(T-t)\le u(R,t),\quad t \in (0,T),$$ where $\alpha=\max\{p,q\}.$ \end{theorem} 

\begin{proof}
Define the functions $J$  as follows:  
$$J(x,t)=u_r(r,t)-\frac{r}{R}e^{u(r,t)}, \quad x\in B_R\times (0,T).$$ 

A direct calculation shows  
\begin{eqnarray*} J_{t}&=&u_{rt}-\frac{r}{R}e^{u}[u_{rr}+\frac{n-1}{r}u_r+\lambda pe^{pu}],\\
J_{r}&=&u_{rr}-\frac{r}{R}e^{u}u_r-\frac{1}{R}e^{u},\\
J_{rr}&=&[u_{rt}-\frac{n-1}{r}u_{rr}+\frac{n-1}{r^2}u_r-\lambda pe^{pu}u_r]\\ &&-\frac{r}{R}[e^{u}u_{rr}+e^{u}u_r^2]-\frac{2}{R}e^{u}u_r.\end{eqnarray*}
From above it follows that \begin{eqnarray*} J_{t}-J_{rr}-\frac{n-1}{r}J_{r}=-\frac{n-1}{r^2}[u_r-\frac{r}{R}e^{u}]+\lambda pe^{pu}[u_r-\frac{r}{R}e^{u}]+\frac{r}{R}e^{u}u_r^2+\frac{2}{R}e^{u}u_r. \end{eqnarray*}
Thus $$J_{t}-\Delta J-bJ=\frac{r}{R}e^{u}u_r^2+\frac{2}{R}e^{u}u_r\ge 0,$$ for $(x,t) \in B_R \times (0,T)\cap \{r>0\},$ 
where $b=[\lambda pe^{pu}-\frac{n-1}{r^2}].$ 

Clearly, from (\ref{FH}), it follows that  $$J(x,0)\ge 0, \quad x \in B_R,$$ and $$J(0,t)=u_r(0,t)\ge 0,~J(R,t)=0\quad t \in (0,T).$$ Since $$\sup_{(0,R)\times (0,t]} b<\infty,\quad \mbox{for}\quad t<T,$$ from above and maximum principle \cite{2}, it follows that
$$J\ge 0,\quad (x,t)\in B_R \times (0,T).$$ Moreover, 
$$\frac{\partial J}{\partial \eta}|_{\partial B_R}\le 0.$$
This means 
$$(u_{rr}-\frac{r}{R}e^{u}u_r-\frac{1}{R}e^{u})|_{\partial B_R}\le 0.$$ 
Thus $$u_t\le (\frac{n-1}{r}u_r+\lambda pe^{pu}+e^{u}u_r+\frac{1}{R}e^{u})|_{\partial B_R}.$$ which implies that
 $$u_t(R,t)\le \frac{n-1}{R}e^{qu(R,t)}+\lambda pe^{pu(R,t)}+e^{(1+q)u(R,t)}+\frac{2}{R}e^{u(R,t)},\quad t \in (0,T).$$

Thus, there exist a constant $C$ such that $$u_t(R,t)\le Ce^{2\alpha u(R,t)},\quad t \in (0,T).$$ Integrate this inequality from $t$ to $T$ and since $u$ blows up at $R,$ it follows 
$$\frac{c}{(T-t)^{\frac{1}{2}}}\le e^{\alpha u(R,t)},\quad t \in (0,T)$$  
or $$\log{c}-\frac{1}{2\alpha}\log (T-t)\le u(R,t),\quad t \in (0,T).$$

\end{proof}

\begin{remark}
From Theorems \ref{pole} and \ref{pole1} we conclude that, when $q>p$ the boundary term plays the dominating role and the lower blow-up rate takes the form: $$\log c-\frac{1}{2q}\log(T-t)\le u(R,t),\quad t \in (0,T),$$ moreover, this estimate is coincident with lower blow-up rate estimate of problem (\ref{B9}), where $\lambda=0,$ which has been considered in \cite{17},
while when $p>q$ the reaction term is dominated and gives the lower blow-up rate as follows $$\log c-\frac{1}{2p}\log(T-t)\le u(R,t),\quad t \in (0,T).$$
\end{remark} 

We next consider the upper bound 
\begin{theorem}\label{poll}
Let $u$ be a solution of problem (\ref{B9}), where $T$ is the blow-up time, $u_0$ satisfies the assumptions (\ref{B4}), (\ref{B4a}) moreover, assume that
\begin{equation}\label{Batre}\Delta u_0+f(u_0)\ge a>0,\quad \mbox{in}~\overline{B}_R.\end{equation}
Then there is a positive constant $C$ such that
\begin{equation}\label{polll} u(R,t)\le \log {C}-\frac{1}{q}\log (T-t),\quad t \in (0,T).\end{equation}
\end{theorem}
\begin{proof}
 
Define the function $J$ as follows $$J(x,t)=u_t(r,t)-\varepsilon u_r(r,t),\quad (x,t) \in B_R \times (0,T).$$ 

Since $u_0(r)$ is bounded in $B_R,$ and by (\ref{Batre}), for some $\varepsilon >0,$ we have $$J(x,0)=\Delta u_0(r)+f(u_0(r))-\varepsilon u_{0r}(r)\ge 0, \quad x \in \overline{B}_R.$$
A simple computation shows  \begin{eqnarray*} J_t&=&u_{rrt}+\frac{n-1}{r}u_{rt}+\lambda pe^{pu}u_t-\varepsilon u_{rt},\\ J_r&=&u_{tr}-\varepsilon u_{rr}, \\ J_{rr}&=&u_{trr}-\varepsilon u_{tr}+\varepsilon\frac{n-1}{r} u_{rr}-\varepsilon \frac{(n-1)}{r^2} u_r +\varepsilon \lambda pe^{pu} u_r.  \end{eqnarray*}   From above, it follows that
$$J_t-J_{rr}-\frac{n-1}{r}J_r-\lambda pe^{pu}J=\varepsilon \frac{(n-1)}{r^2} u_r\ge 0,$$
i.e. $$J_t-\Delta J-\lambda pe^{pu}J\ge 0,\quad (x,t)\in B_R\times (0,T).$$ 

Moreover, \begin{eqnarray*}\frac{\partial J}{\partial \eta}|_{x\in \partial B_R}&=&u_{rt}(R,t)-\varepsilon u_{rr}(R,t)\\  &=&qe^{qu(R,t)}u_t-\varepsilon [u_t(R,t)-\frac{n-1}{r}u_r(R,t) -\lambda e^{pu(R,t)}]\\ &\ge & [qe^{qu(R,t)}-\varepsilon]u_t(R,t) \end{eqnarray*} 
Since, $u_t>0$ in $\overline B_R\times (0,T),$ if follows that  
$$\frac{\partial J}{\partial \eta}\ge 0,\quad \mbox{on}\quad \partial B_R\times (0,T),$$
provided $\varepsilon\le qe^{\{qu_0(R)\}}.$  

Since $e^{pu}$ is bounded on $B_R \times (0,t]$ for $t<T,$ from maximum principle \cite{21} and above, we have $$J\ge 0,\quad (x,t)\in \overline{B}_R \times (0,T).$$
In particular, $J(x,t)\ge 0$ for $x \in \partial B_R,$ that is $$u_t(R,t)\ge \varepsilon u_r(R,t)=\varepsilon e^{qu(R,t)},\quad t \in (0,T).$$
Upon integration the above inequality from $t$ to $T$ and since $u$ blows up at $R,$ it follows that 
$$e^{qu(R,t)}\le \frac{1}{q \varepsilon(T-t)}, \quad t \in (0,T),$$ or $$u(R,t)\le \log {C}-\frac{1}{q}\log (T-t), \quad t \in (0,T).$$
\end{proof}

\begin{remark} The upper blow-up rate estimate for problem (\ref{B9}), which has been derived in Theorem \ref{poll}, is governed by the boundary term even in case $p>q.$
On the other hand, it is known that the upper blow-up bound of problem (\ref{B9}), where $\lambda=0$ (see \cite{17}) takes the form: 
$$u(R,t)\le \log \frac{C}{(T-t)^{\frac{1}{2q}}}.$$
Therefore, we conclude that the presence of the reaction term has an important effect on the upper blow-up rate estimate.   
  \end{remark}
\section{Blow-up Set}
We shall prove in this section that the blow-up to problem (\ref{B9}) occurs only on the boundary, restricting ourselves to the special case $p=q=1$ with some restriction assumption on $\lambda.$\begin{theorem}\label{pol}
Suppose that the function $u(x,t)$ is $C^{2,1}(\overline{B}_R\times[0,T)),$ and satisfies \begin{equation*}\left.  \begin{array}{ll} u_t=\Delta u +\lambda e^{u},& (x,t) \in B_R \times (0,T),\\
u(x,t)\le \log \frac{C}{(T-t)},& (x,t) \in \overline{B}_R \times (0,T),\\
u(x,0)=u_0(x),& x\in \Omega,\end{array}\right\} \end{equation*} where 
\begin{equation}\label{Fc}\lambda [4R^2(n+1)+1]\le \min \left\{ \frac{1}{C}, \frac{4(n+1)}{[R^2+4(n+1)T]}e^{-||u_0||_\infty} \right\},\end{equation} $C<\infty.$
Then for any $0\le a<R,$ there exist a positive constant $A$ such that 
$$u(x,t) \le \log [\frac{1}{A(R^2-r^2)^2}]< \infty \quad \mbox{for}\quad  0\le|x|\le a <R, 0<t<T.$$ 
\end{theorem}
\begin{proof}
Let \begin{eqnarray*} v(x)&=&A(R^2-r^2)^2,\quad r=|x|,\quad 0\le r \le R,\\ z(x,t)=z(r,t)&=&\log \frac{1}{[v(x)+B(T-t)]},\quad \mbox{in}\quad \overline{B}_R \times (0,T),\end{eqnarray*}
where $B>0,A\ge\lambda .$  

A direct calculation shows that \begin{eqnarray*} 
z_t&=&\frac{B}{[v(x)+B(T-t)]},\\ z_r&=&\frac{4rA(R^2-r^2)}{[v(x)+B(T-t)]}, \\ z_{rr}&=&\frac{[v(x)+B(T-t)][4A(R^2-3r^2)]+16A^2r^2(R^2-r^2)^2}{[v(x)+B(T-t)]^2}.\end{eqnarray*}
Thus \begin{eqnarray*}z_t-z_{rr}-\frac{n-1}{r}z_r-\lambda e^z&=&\frac{[B-4A(n-1)(R^2-r^2)-\lambda ][v(x)+B(T-t)]}{[v(x)+B(T-t)]^2}\\ &&-\frac{[4A(R^2-3r^2)][v(x)+B(T-t)] +16Ar^2v(x)}{[v(x)+B(T-t)]^2} \\
&\ge&\frac{[B-4A(n-1)(R^2-r^2)-\lambda -4A(R^2-3r^2)-16Ar^2]v(x)}{[v(x)+B(T-t)]^2} \\ &\ge&\frac{[B-4AR^2n-4AR^2-\lambda]v(x)}{[v(x)+B(T-t)]^2}\\ &\ge&\frac{[B-4AR^2n-4AR^2-A]v(x)}{[v(x)+B(T-t)]^2}\ge 0
\end{eqnarray*} provided $$B\ge A[4R^2(n+1)+1].$$ i.e. $$z_t-\Delta z-\lambda e^z\ge 0,\quad \mbox{in}\quad B_R\times (0,T)$$

Moreover,\begin{equation*}\begin{array}{llll} z(x,0)=\log \frac{1}{[v(x)+BT]}&\ge \log \frac{1}{[AR^4+BT]}&\ge u(x,0),& x \in {B}_R, \\
z(R,t)=\log \frac{1}{B(T-t)}&\ge \log \frac{C}{(T-t)}& \ge u(R,t),&t \in (0,T) \end{array} \end{equation*}
 provided 
$$B\le \min \left\{ \frac{1}{C}, \frac{4(n+1)}{R^2+4(n+1)T}e^{-||u_0||_\infty} \right\}.$$  
From above, and the comparison principle \cite{21}, we obtain $$z(x,t)\ge u(x,t) \quad \mbox{in}\quad B_R \times (0,T).$$ Thus $$u(x,t)\le \log [\frac{1}{A(R^2-r^2)^2}]< \infty \quad \mbox{for}\quad  0\le |x|\le a <R, 0<t<T.$$  
 \end{proof}
\begin{remark} From Theorem \ref{pol} and the upper blow-up rate estimate (\ref{polll}), it followes that,  for the special case of problem (\ref{B9}) ($p=q=1$ and $\lambda$ satisfies (\ref{Fc})), the blow-up occurs only on the boundary. Therefore, we conclude that, the blow-up set of (\ref{B9}), where $\lambda$ is small enough, is the same that of (\ref{B9}), where $\lambda=0$ ( see \cite{17}).
\end{remark}

\end{document}